\pdfoutput=1
\documentclass[11pt]{amsart}
\usepackage{amsmath,amssymb,amsthm,mathtools}
\usepackage[margin=1in]{geometry}
\usepackage{booktabs}
\usepackage[shortlabels]{enumitem}
\usepackage[hidelinks]{hyperref}
\usepackage[capitalise]{cleveref}
\numberwithin{equation}{section}
\newtheorem{theorem}{Theorem}[section]
\newtheorem{proposition}[theorem]{Proposition}
\newtheorem{lemma}[theorem]{Lemma}
\newtheorem{corollary}[theorem]{Corollary}
\theoremstyle{definition}

\theoremstyle{remark}

\newtheorem*{remark*}{Remark}

\newcommand{\F}{\mathbb F}

\newcommand{\R}{\mathbb R}
\newcommand{\C}{\mathbb C}
\newcommand{\Z}{\mathbb Z}
\renewcommand{\epsilon}{\varepsilon}
\renewcommand{\leq}{\leqslant}
\renewcommand{\geq}{\geqslant}
\newcommand{\CP}{\mathbb{CP}}
\newcommand{\Sph}{\mathbb S}
\newcommand{\av}{\operatorname{av}}
\DeclareMathOperator{\rank}{rank}
\DeclareMathOperator{\tr}{tr}

\setlist{itemsep=0.35em,topsep=0.5em}

\title{Polygons and coplanar points on submanifolds of codimension two}
\author{Sohail(Neel) Sarkar}
\address{University of Toronto, Toronto, ON, Canada}
\email{sohail.sarkar@mail.utoronto.ca}
\date{}
\begin{document}
\vspace*{-4mm}

\begin{abstract}
Problem 73 on Green's list of open problems asks whether a smooth submanifold of codimension two in $\R^n$ must meet some affine two-plane in five points when $n$ is large.
In this paper, we answer this question for every $n\geq6$.
Every nonempty relatively open subset of such a submanifold contains the vertices of an affinely regular pentagon, and the bound $n\geq6$ is sharp for affinely regular pentagons.
More generally, for every prime $p\geq5$, the graph of a continuous map from a ball in $\R^{p-1}$ to $\R^2$ contains an affinely regular $p$-gon.
A second argument gives six coplanar points when $n\geq7$.
Both proofs find the points by making the low Fourier modes of the graph map vanish along a circle.
\end{abstract}
\maketitle
\vspace{-1.5\baselineskip}

\section{Introduction}

\subsection{The problem}
A strictly convex hypersurface in $\R^n$ meets every line in at most two points.
Problem 73 on Green's list \cite{Green} asks whether a similar bound can hold one codimension higher:
\begin{quote}
``Let $\Gamma$ be a smooth codimension 2 surface in $\R^n$. Must $\Gamma$
intersect some 2-dimensional plane in 5 points, if $n$ is sufficiently
large?''
\end{quote}
Since the dimensions of $\Gamma$ and of a two-plane add up to $n$, a two-plane in general position meets $\Gamma$ in a discrete set of points.
The question is how many of these points some plane must contain.

The problem is local.
According to Green, $\Gamma$ can be taken to be the graph of a smooth map $(-2,2)^{n-2}\to\R^2$ on the cube $(-1,1)^{n-2}$.
For a submanifold of codimension $d$, he denotes by $f(n,d)$ the number of points through which some affine $d$-plane must pass.
Green observes that $f(n,d)\geq2^d$ for large $n$.
When $d=2$, these four points are the vertices of a parallelogram, so five is the first open case.

In the other direction, $f(n,d)$ is finite.
This was sketched by Thom \cite{Thom} in 1969 and proved in detail by Chaperon and Meyer \cite{ChaperonMeyer} in 2009.
Green also asks whether $f(n,2)$ is bounded independently of $n$.

Throughout the paper, $\langle\cdot,\cdot\rangle$ and $|\cdot|$ denote the Euclidean inner product and norm, all submanifolds are embedded, and a subset of $\Gamma$ is relatively open if it is open in the subspace topology.

\subsection{Main results}
For a submanifold of codimension two in $\R^n$, our results give the following numbers of coplanar points.
In each case the points can be found inside any prescribed nonempty relatively open subset.

\begin{center}
\renewcommand{\arraystretch}{1.25}
\begin{tabular}{@{}lll@{}}
\toprule
Dimension & Points in one affine two-plane & Result\\
\midrule
$n=6$ & $5$, the vertices of an affinely regular pentagon
 & Corollary~\ref{cor:submanifold}(1)\\
$n=7$ & $6$ & Corollary~\ref{cor:submanifold}(2)\\
$n\geq8$ & at least $7$, the vertices of an affinely regular heptagon
 & Corollary~\ref{cor:submanifold}(1)\\
$n\geq p+1$ & $p$, for every prime $p\geq5$ & Corollary~\ref{cor:submanifold}(1)\\
\bottomrule
\end{tabular}
\end{center}

An \emph{affinely regular $p$-gon} is the image of the vertices of a regular $p$-gon under an injective affine map.
Its vertices can be written as
\[
 a+b\cos(2\pi j/p)+c\sin(2\pi j/p),\qquad 0\leq j<p,
\]
where $b$ and $c$ are linearly independent.
It is \emph{regular} if, in addition, $b$ and $c$ are orthogonal and have the same length.

We state both main results for graphs, where the hypotheses are weakest, and deduce the statements for submanifolds in Corollary~\ref{cor:submanifold}.

\begin{theorem}[Polygons on graphs]\label{thm:polygon}
Let $p\geq5$ be prime, let $B\subset\R^m$ be an open ball with centre $x_0$, let $g:B\to\R^2$ be continuous, and let $r>0$ be such that the closed ball of radius $r$ about $x_0$ lies in $B$.
\begin{enumerate}
\item If $m\geq p-1$, there is a regular $p$-gon $y_0,\ldots,y_{p-1}$,
inscribed in a circle of radius $r$ centred at $x_0$, such that the points $(y_j,g(y_j))$, $0\leq j<p$, form an affinely regular $p$-gon.
\item If $m\geq p$, the polygon $y_0,\ldots,y_{p-1}$ can be chosen so
that these points form a regular $p$-gon.
\end{enumerate}
\end{theorem}

Only the continuity of $g$ is used in the proof.
The dimension bound in part (1) cannot be lowered for pentagons.

\begin{proposition}[Sharpness for pentagons]\label{prop:sharp}
The smooth submanifold
\[
 \{(x_1,x_2,x_3,x_1^2,x_2^2):(x_1,x_2,x_3)\in\R^3\}\subset\R^5
\]
contains no affinely regular pentagon.
\end{proposition}

Since this graph contains straight lines, the example does not bound the number of coplanar points in $\R^5$.
Our second result counts coplanar points without prescribing their shape.

\begin{theorem}[Six coplanar points]\label{thm:six}
Let $m\geq5$, let $U\subset\R^m$ be a nonempty open set, and let $F,G:U\to\R$, where $F$ is of class $C^2$ and $G$ is continuous.
Then the graph $\{(x,F(x),G(x)):x\in U\}$ contains six distinct points lying in one affine two-plane.
\end{theorem}

\begin{corollary}[Submanifolds of codimension two]\label{cor:submanifold}
Let $\Gamma\subset\R^n$ be a nonempty $C^1$ submanifold of codimension two, and let $\mathcal O\subset\Gamma$ be nonempty and relatively open.
\begin{enumerate}
\item If $p\geq5$ is prime, then $\mathcal O$ contains an affinely
regular $p$-gon when $n\geq p+1$, and a regular $p$-gon when $n\geq p+2$.
\item If $\Gamma$ is of class $C^2$ and $n\geq7$, then $\mathcal O$
contains six distinct points lying in one affine two-plane.
\end{enumerate}
\end{corollary}

\begin{proof}
Choose a point of $\mathcal O$, and apply a rigid motion of $\R^n$ that moves this point to the origin and its tangent space to $\R^{n-2}\times\{0\}$.
By the inverse function theorem, applied to the orthogonal projection onto $\R^{n-2}\times\{0\}$, a neighbourhood of the origin in the image of $\Gamma$ is the graph of a map from an open ball about the origin in $\R^{n-2}$ to $\R^2$.
This map is of class $C^1$, and of class $C^2$ in part (2).
After shrinking the ball, we may assume that the graph lies in the image of $\mathcal O$.
Theorems~\ref{thm:polygon} and~\ref{thm:six} with $m=n-2$ now give the required points.
Rigid motions preserve distinct points, affine two-planes, and both kinds of polygon.
\end{proof}

Taking $p=5$, five coplanar points exist in every dimension $n\geq6$, which answers Green's question for all $n\geq6$.
In Green's notation, we obtain $f(n,2)\geq p$ for every prime $p$ with $5\leq p\leq n-1$.

\subsection{Outline of the proofs}
Both proofs seek points of a plane at which the graph map agrees with an affine map.
Indeed, let $Q\subset\R^m$ be an affine two-plane, and let $y_1,\ldots,y_k$ be points of $Q$.
If a map $\phi$ with values in $\R^2$ agrees at these points with an affine map $\ell:Q\to\R^2$, then the graph points $(y_j,\phi(y_j))$ lie in the affine two-plane $\{(y,\ell(y)):y\in Q\}$.

When the points lie on a circle in $Q$, this condition can be expressed in terms of Fourier modes: the restriction of an affine function to a circle contains only the modes $0$ and $\pm1$.

For Theorem~\ref{thm:polygon}, we evaluate $g$ at the vertices of a regular $p$-gon of radius $r$ centred at $x_0$.
These values agree with an affine map on the plane of the polygon exactly when their discrete Fourier coefficients of orders $2,\ldots,p-2$ vanish.
This gives $p-3$ complex conditions on the orthonormal frame spanning the polygon, and the frame is determined only up to a rotation through $2\pi/p$.
A cohomology computation on the quotient of the frame space, Lemma~\ref{lem:zero}, shows that these conditions can be satisfied when $p-3\leq m-2$.
A regular polygon requires, in addition, that the coefficient of order $p-1$ vanishes, and therefore one more dimension.

The proof of Theorem~\ref{thm:six} instead uses the Hessian of $F$.
If the Hessian of $F$ is never definite on a three-dimensional subspace, then its rank is at most four, and the constant rank theorem gives a segment on which $F$ is affine.
The graph over this segment lies in a plane.

Otherwise, on a suitable three-dimensional slice, $F$ minus its tangent plane has a strict local minimum, and a small level set $\Sigma$ of this function is a closed surface around the minimum.
On every section of $\Sigma$ by a plane $E$ through the minimum, $F$ agrees with an affine function.
It remains to make $G$ agree with an affine function at six points of one such section.

Subtracting a suitable affine function eliminates the modes $0$ and $\pm1$ of $G$ along every section.
The mode $\pm2$ imposes two real conditions, and the planes $E$ form a two-parameter family.
A hairy ball argument, Lemma~\ref{lem:tensor}, shows that these two conditions can be met.
Once the modes $0$, $\pm1$ and $\pm2$ vanish, the Sturm--Hurwitz theorem gives six zeros.

\subsection{Earlier work}
Knaster's problem asks for configurations of points on a sphere at which every continuous map takes equal values after a suitable rotation.
In 1990, Makeev \cite[p.~435, Remark~3]{Makeev} proved that if $p$ is an odd prime and $2k-2>2(p-1)$, then every continuous map $\Sph^{k-1}\to\R^2$ takes the same value at all vertices of some regular $p$-gon inscribed in a great circle.
Applied with $k=n-2$ to a small sphere in a graph chart of $\Gamma$, this gives a regular $p$-gon on $\Gamma$ when $n\geq p+3$.
In particular, Makeev's theorem answers Green's question for $n\geq8$ and shows that $f(n,2)\to\infty$.
Equal graph values force all $p-1$ nonzero discrete Fourier coefficients to vanish.

In 2010, Karasev and Volovikov \cite[Theorem~5]{KarasevVolovikov} proved a version that permits a deviation from constancy in one fixed direction.
Take the orbit point in their theorem in the two-dimensional summand of the regular representation of $\Z/p\Z$ on which the generator rotates through $2\pi/p$, and take the fixed direction to be $(\cos(2\pi j/p),\sin(2\pi j/p))_j$.
Their theorem then gives graph values of the form $(a+t\cos(2\pi j/p), b+t\sin(2\pi j/p))$ on a regular $p$-gon, and hence a regular $p$-gon on $\Gamma$ when $n\geq p+2$.
Thus part (2) of Theorem~\ref{thm:polygon} also follows from their theorem.
Part (1) asks only for affine regularity and therefore requires one dimension less.
Proposition~\ref{prop:sharp} shows that this threshold is sharp for pentagons.

The cohomology class used in Lemma~\ref{lem:zero} is the one used by Makeev in the proof of his Lemma~1 \cite[\S1]{Makeev}.
Makeev also studied polygons inscribed in planar convex figures and in plane sections of convex bodies \cite{Makeev1995}, in continuous fields of curves lying in the planes of the tautological bundle over the Grassmannian of oriented two-planes \cite{Makeev2004}, and in continuous fields of convex figures \cite{Makeev2007}.
In these settings the polygon lies in the plane of a figure, whereas in our setting it lies on a submanifold of codimension two.

To the author's knowledge, part (1) of Theorem~\ref{thm:polygon}, Proposition~\ref{prop:sharp} and Theorem~\ref{thm:six} are new.
These results improve the preceding bounds in dimension seven and whenever $n\geq6$ and $n-1$ is prime.
In dimension seven, Theorem~\ref{thm:six} gives six points, one more than the polygon results give in that dimension, and its proof does not use the topology of frame spaces.

\subsection{Organization}
Section~\ref{sec:polygons} contains the proofs of Theorem~\ref{thm:polygon} and Proposition~\ref{prop:sharp}.
Section~\ref{sec:six} contains the proof of Theorem~\ref{thm:six}.

\section{Proof of Theorem~\ref{thm:polygon} and
Proposition~\ref{prop:sharp}}\label{sec:polygons}

\subsection{A zero of an equivariant map}
The topological input for Theorem~\ref{thm:polygon} is the following lemma.
Denote by $V_2(\R^m)$ the space of linear isometries $U:\R^2\to\R^m$; its points are ordered pairs of orthonormal vectors.
Let $C_p$ be the cyclic group of order $p$, and put $\zeta=e^{2\pi i/p}$.
A generator of $C_p$ acts on $V_2(\R^m)$ by $U\mapsto U\circ R$, where $R$ is the rotation of $\R^2$ through the angle $2\pi/p$.
This action is free.

\begin{lemma}\label{lem:zero}
Let $m\geq3$, let $p$ be an odd prime, and let $k_1,\ldots,k_s\in\{1,\ldots,p-1\}$, where $1\leq s\leq m-2$.
Every continuous map
\[
 \Phi=(\Phi_1,\ldots,\Phi_s):V_2(\R^m)\longrightarrow\C^s
 \qquad\text{with}\qquad
 \Phi_i(U\circ R)=\zeta^{k_i}\Phi_i(U)\quad(1\leq i\leq s)
\]
has a zero.
\end{lemma}

\begin{proof}
The quotient of $V_2(\R^m)$ by all rotations of the frame is the Grassmannian of oriented two-planes in $\R^m$.
If $U$ has columns $u$ and $v$, the map $U\mapsto[u+iv]$ identifies this Grassmannian with the smooth complex quadric
\[
 Q=\{[z_1:\cdots:z_m]\in\CP^{m-1}:z_1^2+\cdots+z_m^2=0\}.
\]
Indeed, the real and imaginary parts of a nonzero vector $z$ with $\sum_iz_i^2=0$ are orthogonal and have the same positive length, and multiplying $z$ by a nonzero complex number changes the corresponding frame only by a rotation; see also \cite[Section~1, Example~3]{Gadbled}.

Let $L$ be the restriction to $Q$ of the tautological line bundle of $\CP^{m-1}$, and denote $e=c_1(L)\in H^2(Q;\Z)$.
Using the unit vector $(u+iv)/\sqrt2$ in each fibre, we identify $X=V_2(\R^m)$ with the unit circle bundle of $L$.
The class $-e$ is the hyperplane class.
Moreover, $Q$ has degree two, since $m-2$ generic hyperplanes cut $\CP^{m-1}$ in a projective line on which the quadratic form has two simple zeros.
Hence
\begin{equation}\label{eq:degree}
 \langle e^{m-2},[Q]\rangle=(-1)^{m-2}\,2.
\end{equation}

Put $Y=X/C_p$, and let $\pi:Y\to Q$ be the projection.
Taking the $p$th tensor power of a unit vector identifies $Y$ with the unit circle bundle of $L^{\otimes p}$, whose Euler class is $pe$.
With coefficients in $\F_p=\Z/p\Z$, the Gysin sequence \cite[Section~3.2]{HatcherVB} therefore contains the exact segment
\[
 H^{j-2}(Q;\F_p)\xrightarrow{\ \smile\, pe\,=\,0\ }
 H^j(Q;\F_p)\xrightarrow{\ \pi^*\ }H^j(Y;\F_p),
\]
so $\pi^*$ is injective.
Since $p$ is odd, \eqref{eq:degree} shows that $e^{m-2}$ is nonzero modulo $p$.
Consequently, $e^s$ is nonzero modulo $p$ for every $s\leq m-2$, and so is $\pi^*(e^s)$.

Now consider the complex vector bundle
\[
 E=\bigoplus_{i=1}^s\pi^*(L^{\otimes k_i})\longrightarrow Y.
\]
Under $U\mapsto U\circ R$, the vector $u+iv$ is multiplied by $\zeta^{-1}$, so $\Phi_i(U)(u+iv)^{\otimes k_i}$ depends only on the class of $U$ in $Y$.
Thus $\Phi$ defines a continuous section of $E$, which vanishes exactly where $\Phi$ vanishes.
By the Whitney sum formula, the top Chern class of $E$, reduced modulo $p$, is
\[
 c_s(E)=\Bigl(\prod_{i=1}^sk_i\Bigr)\pi^*(e^s)\neq0,
\]
because each $k_i$ is prime to $p$.
On the other hand, a nowhere-zero section of a complex vector bundle of rank $s$ spans a trivial line subbundle, whose orthogonal complement has rank $s-1$, and this would force $c_s(E)=0$.
Hence $\Phi$ has a zero.
\end{proof}

\subsection{Proof of Theorem~\ref{thm:polygon}}
We identify $\R^2$ with $\C$.
For $U\in V_2(\R^m)$, put
\[
 v_j=(\cos(2\pi j/p),\sin(2\pi j/p))\in\R^2,\qquad
 y_j(U)=x_0+rUv_j\qquad(0\leq j<p),
\]
so that $y_0(U),\ldots,y_{p-1}(U)$ is a regular $p$-gon inscribed in the circle of radius $r$ about $x_0$ in the plane $x_0+U(\R^2)$.
Every such polygon arises in this way.
We define the discrete Fourier coefficients
\[
 \Phi_k(U)=\frac1p\sum_{j=0}^{p-1}\zeta^{-kj}g\bigl(y_j(U)\bigr)
 \qquad(0\leq k<p).
\]
Since $Rv_j=v_{j+1}$, reindexing the sum gives $\Phi_k(U\circ R)=\zeta^k\Phi_k(U)$, and Fourier inversion gives
\begin{equation}\label{eq:inversion}
 g\bigl(y_j(U)\bigr)=\sum_{k=0}^{p-1}\Phi_k(U)\,\zeta^{kj}
 \qquad(0\leq j<p).
\end{equation}

We first prove part (1).
Apply Lemma~\ref{lem:zero} to $(\Phi_2,\ldots,\Phi_{p-2})$.
This map has $p-3\geq2$ components, with weights in $\{1,\ldots,p-1\}$, and $p-3\leq m-2$ because $m\geq p-1$.
At a common zero $U$, with columns $u$ and $v$, identity \eqref{eq:inversion} becomes
\[
 g(y_j)=\alpha+\beta\zeta^j+\gamma\zeta^{-j}
 =\alpha+(\beta+\gamma)\cos(2\pi j/p)+i(\beta-\gamma)\sin(2\pi j/p)
\]
for some $\alpha,\beta,\gamma\in\C$.
Since $y_j=x_0+r\cos(2\pi j/p)u +r\sin(2\pi j/p)v$, the graph points are
\[
 (x_0,\alpha)+(ru,\beta+\gamma)\cos(2\pi j/p)
 +(rv,i(\beta-\gamma))\sin(2\pi j/p).
\]
The two direction vectors are linearly independent, because their components in $\R^m$ are $ru$ and $rv$.
Hence the graph points form an affinely regular $p$-gon.

For part (2), we apply the lemma to $(\Phi_2,\ldots,\Phi_{p-1})$, which has $p-2\leq m-2$ components.
At a common zero, $g(y_j)=\alpha+\beta \zeta^j$, and the direction vectors become $(ru,\beta)$ and $(rv,i\beta)$.
Their inner product is $r^2\langle u,v\rangle+ \operatorname{Re}(\beta\,\overline{i\beta})=0$, and both have squared length $r^2+|\beta|^2$.
Hence the graph points form a regular $p$-gon.\qed

\subsection{Proof of Proposition~\ref{prop:sharp}}
Suppose that the submanifold contains an affinely regular pentagon with vertices $A+B\cos(2\pi j/5)+C\sin(2\pi j/5)$, $0\leq j<5$, where $B,C\in\R^5$ are linearly independent.
For $t=1,2$, its $t$th coordinate is
\[
 x_t(j)=a_t+b_t\cos(2\pi j/5)+c_t\sin(2\pi j/5),
\]
and its $(t+3)$th coordinate is $x_t(j)^2$.
As a function on $\Z/5\Z$, every coordinate of an affinely regular pentagon has its discrete Fourier coefficients supported on the frequencies $0$, $1$ and $4$.

Since $b_t\cos\theta+c_t\sin\theta=\tfrac12(b_t-ic_t)e^{i\theta} +\tfrac12(b_t+ic_t)e^{-i\theta}$, the coefficient of frequency $2$ in $x_t(j)^2$ equals $\tfrac14(b_t-ic_t)^2$.
Here no other term contributes, because the frequencies $0,\pm1,\pm2$ are distinct modulo five.
This coefficient must therefore vanish, and since $b_t$ and $c_t$ are real, $b_t=c_t=0$.

Consequently, the first two coordinates of the pentagon are constant, and so are the last two.
Thus $B$ and $C$ are multiples of the third basis vector, which contradicts their linear independence.\qed

\section{Proof of Theorem~\ref{thm:six}}\label{sec:six}

\subsection{Notation}
We write $\Sph^2=\{v\in\R^3:|v|=1\}$.
For $v\in\Sph^2$, the plane $v^\perp$ is the tangent plane of $\Sph^2$ at $v$, and $\pi_v$ denotes the orthogonal projection of $\R^3$ onto $v^\perp$.
For a two-dimensional linear subspace $E\subset\R^3$, let $S(E)=E\cap\Sph^2$, let $ds$ be arclength on $S(E)$, and let $I_E$ be the identity map of $E$.
For a continuous function $\phi$ on $S(E)$, with values in $\R$, in $E$, or in the linear maps of $E$, we denote its angular average by
\[
 \av_E[\phi]=\frac1{2\pi}\int_{S(E)}\phi(u)\,ds(u).
\]
Finally, for $u\in\R^3$, we write $u\otimes u$ for the linear map $z\mapsto\langle u,z\rangle u$.

\subsection{Six zeros on a circle}
The Sturm--Hurwitz theorem states that a continuous function on the circle which is orthogonal to all trigonometric polynomials of degree less than $N$ has at least $2N$ sign changes, unless it vanishes identically; see Tabachnikov \cite{Tabachnikov} for this theorem, its converse, and further references.
We need only the case $N=3$, which has a short elementary proof.

\begin{lemma}\label{lem:zeros}
Let $h:\R/2\pi\Z\to\R$ be continuous, and suppose that
\[
 \int_0^{2\pi}h(\theta)\,d\theta=0
 \qquad\text{and}\qquad
 \int_0^{2\pi}h(\theta)\cos(j\theta)\,d\theta
 =\int_0^{2\pi}h(\theta)\sin(j\theta)\,d\theta=0
 \quad(j=1,2).
\]
Then $h$ has at least six distinct zeros.
\end{lemma}

\begin{proof}
Suppose that $h$ has at most five zeros.
Then its zero set is finite, and by the intermediate value theorem $h$ has constant sign on each component of the complement.
We say that $h$ changes sign at a zero if it takes opposite signs immediately before and after it.
Going once around the circle returns $h$ to its initial sign, so the number of sign changes is even, say $2k$, and $2k\leq4$.
Let $\theta_1,\ldots,\theta_{2k}$ be the zeros at which $h$ changes sign, and put
\[
 q(\theta)=\prod_{j=1}^{2k}\sin\frac{\theta-\theta_j}{2},
\]
with $q=1$ if $k=0$.

The identity
\[
 \sin\frac{\theta-\alpha}{2}\,\sin\frac{\theta-\beta}{2}
 =\frac12\Bigl(\cos\frac{\alpha-\beta}{2}
 -\cos\Bigl(\theta-\frac{\alpha+\beta}{2}\Bigr)\Bigr),
\]
applied to the factors in pairs, shows that $q$ is a $2\pi$-periodic trigonometric polynomial of degree at most $k\leq2$.
As a function on $\R$, the $j$th factor vanishes exactly on $\theta_j+2\pi\Z$ and changes sign at each of these points.
Since the $\theta_j$ are distinct modulo $2\pi$, the function $q$ changes sign at each $\theta_j$ and has no other zeros on the circle.

The points $\theta_1,\ldots,\theta_{2k}$ divide the circle into $2k$ arcs (for $k=0$, we take the whole circle).
On each arc, $h$ does not change sign, so $h\geq0$ or $h\leq0$ there, while $q$ has a constant strict sign.
Both signs change across every $\theta_j$.
After replacing $q$ by $-q$ if necessary, we obtain $hq\geq0$ on the whole circle, with strict inequality outside a finite set.
Hence
\[
 \int_0^{2\pi}h(\theta)q(\theta)\,d\theta>0.
\]
On the other hand, $q$ is a linear combination of $1$, $\cos\theta$, $\sin\theta$, $\cos2\theta$ and $\sin2\theta$, so the hypotheses imply that this integral vanishes, a contradiction.
\end{proof}

\subsection{Trace-zero fields on the sphere}
For a continuous function $h$ on a circle $S(E)$, the operator $\av_E\bigl[h(u)\bigl(u\otimes u-\tfrac12I_E\bigr)\bigr]$ on $E$ is symmetric and has trace zero.
In an orthonormal basis of $E$, its entries are, up to a factor $\tfrac12$, the averages of $h$ against $\cos2\theta$ and $\sin2\theta$.
Thus the second Fourier mode of $h$ is a symmetric trace-zero operator on $E$, and it vanishes exactly when this operator does.

\begin{lemma}\label{lem:tensor}
For each $v\in\Sph^2$, let $T_v$ be a symmetric linear map of $v^\perp$ with trace zero.
If the map $v\mapsto T_v\pi_v$ from $\Sph^2$ to the space of $3\times3$ matrices is continuous, then $T_v=0$ for some $v\in\Sph^2$.
\end{lemma}

\begin{proof}
Suppose that $T_v\neq0$ for every $v$.
A nonzero symmetric trace-zero map of a plane has eigenvalues $\lambda$ and $-\lambda$ for some $\lambda>0$.
Here $\lambda(v)=\bigl(\tfrac12\tr\bigl((T_v\pi_v)^2\bigr) \bigr)^{1/2}$ depends continuously on $v$, and so does
\[
 P_v=\frac{T_v\pi_v+\lambda(v)\pi_v}{2\lambda(v)},
\]
the orthogonal projection of $\R^3$ onto the eigenline $L_v\subset v^\perp$ of $T_v$ for the eigenvalue $\lambda(v)$.

Consider the map
\[
 \widetilde S=\{(v,w)\in\Sph^2\times\Sph^2:w\in L_v\}\longrightarrow\Sph^2,
 \qquad(v,w)\longmapsto v.
\]
Fix $v_0\in\Sph^2$ and a unit vector $w_0\in L_{v_0}$.
For $v$ in a neighbourhood $N$ of $v_0$, we have $P_vw_0\neq0$, and the fibre over $v$ consists of the two points $(v,\pm P_vw_0/|P_vw_0|)$.
Therefore the part of $\widetilde S$ over $N$ is the disjoint union of the graphs of the continuous maps $v\mapsto\pm P_vw_0/|P_vw_0|$, and $\widetilde S\to\Sph^2$ is a two-sheeted covering map.

The sphere is simply connected and locally path-connected.
The lifting criterion \cite[Proposition~1.33]{Hatcher}, applied to the identity map of $\Sph^2$, therefore gives a continuous section $v\mapsto(v,w(v))$.
Then $w$ is a continuous unit tangent vector field on $\Sph^2$, which contradicts the hairy ball theorem \cite[Theorem~2.28]{Hatcher}.
Indeed, the map
\[
 (v,t)\longmapsto\cos(\pi t)\,v+\sin(\pi t)\,w(v),\qquad0\leq t\leq1,
\]
would be a homotopy from the identity of $\Sph^2$ to the antipodal map, whose degrees are $1$ and~$-1$.
\end{proof}

\subsection{Plane sections of a star-shaped surface}
For a continuous function $\rho:\Sph^2\to(0,\infty)$, the set $\{\rho(u)u:u\in\Sph^2\}$ is the boundary of the star-shaped body with radial function $\rho$.
The next proposition handles the case of Theorem~\ref{thm:six} in which the Hessian of $F$ is definite on a three-dimensional subspace.

\begin{proposition}\label{prop:radial}
Let $\rho:\Sph^2\to(0,\infty)$ be continuous, and put $\Sigma=\{\rho(u)u:u\in\Sph^2\}$.
For every continuous function $G:\Sigma\to\R$, there are a two-dimensional linear subspace $E\subset\R^3$ and an affine function $\ell:E\to\R$ such that $G(x)=\ell(x)$ at six distinct points $x\in\Sigma\cap E$.
\end{proposition}

\begin{proof}
Put $H(u)=G(\rho(u)u)$, which is a continuous function on $\Sph^2$, and fix a two-dimensional linear subspace $E\subset\R^3$.
We first subtract from $H$ an affine function of the point $\rho(u)u$, chosen so that the difference has zero mean and zero first Fourier mode on $S(E)$.
The symmetric linear map
\[
 M_E=\av_E\bigl[\rho(u)\,u\otimes u\bigr]
\]
of $E$ is positive definite, since $\langle M_Ez,z\rangle= \av_E[\rho(u)\langle u,z\rangle^2]>0$ for $z\in E\setminus\{0\}$.
We define
\[
 b_E=M_E^{-1}\av_E\bigl[H(u)\,u\bigr]\in E,\qquad
 a_E=\av_E\bigl[H(u)-\rho(u)\langle b_E,u\rangle\bigr],
\]
and, for $u\in S(E)$,
\[
 h_E(u)=H(u)-a_E-\langle b_E,\rho(u)u\rangle.
\]
Since $\av_E[u]=0$ and $\av_E[\rho(u)\langle b_E,u\rangle u]=M_Eb_E$, these choices give
\begin{equation}\label{eq:low-modes}
 \av_E[h_E]=0\qquad\text{and}\qquad\av_E[h_E(u)\,u]=0.
\end{equation}

The correction $a_E+\langle b_E,x\rangle$ is affine in the point $x=\rho(u)u$ of $\Sigma$, not in the direction $u$.
Subtracting a function of the form $a+\langle b,u\rangle$ would eliminate the same two modes, but it would not be the restriction of an affine function of $x$, and the six points found below need not be coplanar.

It remains to eliminate the second mode.
The linear map
\[
 T_E=\av_E\bigl[h_E(u)\bigl(u\otimes u-\tfrac12I_E\bigr)\bigr]
\]
of $E$ is symmetric and has trace zero, because $\tr(u\otimes u)=|u|^2=1$ and $\tr I_E=2$.
We claim that the maps $T_{v^\perp}$, for $v\in\Sph^2$, satisfy the hypothesis of Lemma~\ref{lem:tensor}.

Fix $v_0\in\Sph^2$ and a basis $c_1,c_2$ of $v_0^\perp$.
For $v$ in a neighbourhood $N$ of $v_0$ in $\Sph^2$, the vectors $\pi_vc_1,\pi_vc_2$ form a basis of $v^\perp$, and the Gram--Schmidt process turns them into an orthonormal basis $e_1(v),e_2(v)$ of $v^\perp$ that depends continuously on $v$.
With $u=\cos\theta\,e_1(v)+\sin\theta\,e_2(v)$, we have
\[
 \av_{v^\perp}[\phi]=\frac1{2\pi}\int_0^{2\pi}\phi\bigl(\cos\theta\,
 e_1(v)+\sin\theta\,e_2(v)\bigr)\,d\theta,
\]
and the integral of a function that is continuous in $(v,\theta)$ depends continuously on $v$.
In the basis $e_1(v),e_2(v)$, the map $M_{v^\perp}$ is a positive definite $2\times2$ matrix depending continuously on $v$, and so is its inverse.
Hence $b_{v^\perp}$ and $a_{v^\perp}$ depend continuously on $v$, the function $(v,\theta)\mapsto h_{v^\perp}(u)$ is continuous, and finally
\[
 T_{v^\perp}\pi_v=\frac1{2\pi}\int_0^{2\pi}h_{v^\perp}(u)\bigl(u\otimes
 u-\tfrac12\pi_v\bigr)\,d\theta
\]
depends continuously on $v\in N$.
Lemma~\ref{lem:tensor} therefore gives a plane $E$ with $T_E=0$.

In an orthonormal basis of this plane $E$, we write $u=(\cos\theta,\sin\theta)$.
Then
\[
 u\otimes u-\tfrac12I_E=\frac12
 \begin{pmatrix}
 \cos2\theta&\sin2\theta\\
 \sin2\theta&-\cos2\theta
 \end{pmatrix},
\]
so $T_E=0$ means that the averages of $h_E$ against $\cos2\theta$ and $\sin2\theta$ vanish.
By \eqref{eq:low-modes}, the averages of $h_E$ against $1$, $\cos\theta$ and $\sin\theta$ vanish as well.
Lemma~\ref{lem:zeros} now gives six distinct zeros $u_1,\ldots,u_6\in S(E)$ of $h_E$.
The points $x_j=\rho(u_j)u_j$ lie in $\Sigma\cap E$, they are distinct because the $u_j$ are distinct and $\rho>0$, and the affine function $\ell(x)=a_E+\langle b_E,x\rangle$ on $E$ satisfies
\[
 G(x_j)=H(u_j)=\ell(x_j)\qquad(1\leq j\leq6).\qedhere
\]
\end{proof}

\subsection{Hessians of small rank}
For the degenerate case of Theorem~\ref{thm:six}, we use the following lemma.

\begin{lemma}\label{lem:hessian}
Let $W\subset\R^m$ be a nonempty open set, and let $F:W\to\R$ be of class $C^2$.
Suppose that the Hessian $D^2F$ has the same rank $r<m$ at every point of $W$.
Then there are $a\in W$, a nonzero vector $v\in\R^m$ and $\epsilon>0$ such that, for $|t|<\epsilon$, the point $a+tv$ lies in $W$ and $F(a+tv)$ is an affine function of $t$.
\end{lemma}

\begin{proof}
If $r=0$, the gradient $\nabla F$ is constant on every open ball in $W$, so $F$ is affine there.
Suppose now that $r\geq1$.
The map $\nabla F:W\to\R^m$ is of class $C^1$, and its derivative $D^2F$ has constant rank $r$.
By the constant rank theorem, there are an open set $W_0\subset W$, open boxes $B_1\subset\R^r$ and $B_2\subset\R^{m-r}$, a $C^1$ diffeomorphism $x:B_1\times B_2\to W_0$, and a $C^1$ diffeomorphism $\psi$ from an open set containing $\nabla F(W_0)$ onto an open subset of $\R^m$, such that $\psi(\nabla F(x(s,t)))=(s,0)$ for all $(s,t)\in B_1\times B_2$.
Thus $\nabla F(x(s,t))=p(s)$, where $p(s)=\psi^{-1}(s,0)$, and $Dp(s)$ has rank $r$ because $\psi^{-1}$ is a diffeomorphism.

For $1\leq k\leq m-r$, the chain rule gives
\[
 \partial_{t_k}\bigl[F(x(s,t))-\langle p(s),x(s,t)\rangle\bigr]
 =\bigl\langle\nabla F(x(s,t))-p(s),\,\partial_{t_k}x(s,t)\bigr\rangle=0.
\]
Since $B_2$ is connected, $F(x(s,t))-\langle p(s),x(s,t)\rangle=A(s)$ for a $C^1$ function $A$ on $B_1$.
Differentiating this identity with respect to $s_j$, and using $\nabla F(x(s,t))=p(s)$ once more, we find
\begin{equation}\label{eq:affine-fibre}
 \partial_{s_j}A(s)=-\bigl\langle\partial_{s_j}p(s),\,x(s,t)\bigr\rangle
 \qquad(1\leq j\leq r,\ t\in B_2).
\end{equation}

Now fix $s\in B_1$.
The vectors $\partial_{s_1}p(s),\ldots, \partial_{s_r}p(s)$ are linearly independent, so \eqref{eq:affine-fibre} places the points $x(s,t)$, $t\in B_2$, in an affine subspace $K$ of dimension $m-r\geq1$.
The map $t\mapsto x(s,t)$ is an injective immersion of the $(m-r)$-dimensional box $B_2$ into $K$, so its image is open in $K$ by the inverse function theorem.
On this image, $F(x)= \langle p(s),x\rangle+A(s)$ is an affine function of $x$, so $F$ is affine on every line segment in the image.
\end{proof}

\subsection{Proof of Theorem~\ref{thm:six}}
Either at every $x\in U$ the Hessian $D^2F(x)$ has at most two positive and at most two negative eigenvalues, or at some $a\in U$ it has three eigenvalues of the same strict sign.

\medskip
\noindent\textit{Case 1: at every $x\in U$, the Hessian $D^2F(x)$ has at most two positive and at most two negative eigenvalues.} Then $\rank D^2F(x)\leq4$ for all $x\in U$.
Let $r$ be the largest value of this rank, attained at a point $x_0$.
If $r\geq1$, some $r\times r$ minor of $D^2F(x_0)$ is nonzero.
It stays nonzero on a neighbourhood $W\subset U$ of $x_0$, and by maximality $D^2F$ has rank exactly $r$ on $W$.
If $r=0$, we take $W=U$.

Since $r\leq4<m$, Lemma~\ref{lem:hessian} gives $a\in W$, a nonzero $v\in\R^m$, $\epsilon>0$ and $\alpha,\beta\in\R$ with $F(a+tv)=\alpha+\beta t$ for $|t|<\epsilon$.
The graph points
\[
 \bigl(a+tv,\ \alpha+\beta t,\ G(a+tv)\bigr)\qquad(|t|<\epsilon)
\]
lie in the set $\{(a+tv,\ \alpha+\beta t,\ z):t,z\in\R\}$, which is an affine two-plane because $v\neq0$.
Six distinct values of $t$ give six distinct points.

\medskip
\noindent\textit{Case 2: at some $a\in U$, the Hessian $D^2F(a)$ has three eigenvalues of the same strict sign.} Let $B:\R^3\to\R^m$ be a linear isometry onto the span of three corresponding orthonormal eigenvectors, and choose $\sigma\in\{-1,1\}$ such that $\sigma B^{\mathsf T}D^2F(a)B$ is positive definite.
For $y$ near $0\in\R^3$, put
\[
 F_0(y)=\sigma\bigl(F(a+By)-F(a)-\langle\nabla F(a),By\rangle\bigr),
 \qquad G_0(y)=G(a+By).
\]
Then $F_0(0)=0$, $\nabla F_0(0)=0$, and $D^2F_0(0)=\sigma B^{\mathsf T}D^2F(a)B$ is positive definite.
By continuity, there are $R>0$ and $\lambda>0$ such that $a+By\in U$ and $\langle D^2F_0(y)z,z\rangle\geq\lambda|z|^2$ for all $z\in\R^3$ whenever $|y|\leq R$.

For $u\in\Sph^2$, the function $\varphi_u(r)=F_0(ru)$ on $[0,R]$ satisfies $\varphi_u(0)=\varphi_u'(0)=0$ and $\varphi_u''(r)=\langle D^2F_0(ru)u,u\rangle\geq\lambda$.
Therefore
\[
 \varphi_u'(r)\geq\lambda r
 \qquad\text{and}\qquad
 \varphi_u(r)\geq\tfrac12\lambda r^2
 \qquad(0\leq r\leq R).
\]
Fix $c$ with $0<c<\frac12\lambda R^2$.
Then $\varphi_u$ increases strictly on $[0,R]$ from $0$ to a value larger than $c$, so there is exactly one $\rho(u)\in(0,R)$ with $F_0(\rho(u)u)=c$.
The implicit function theorem, applied to $(w,r)\mapsto F_0(rw)-c$ near $(u,\rho(u))$, where the derivative in $r$ equals $\varphi_u'(\rho(u))\geq\lambda\rho(u)>0$, together with the uniqueness of $\rho(u)$, shows that $\rho$ is continuous.

Proposition~\ref{prop:radial}, applied to $\Sigma=\{\rho(u)u:u\in\Sph^2\}$ and to the restriction of $G_0$ to $\Sigma$, gives a two-dimensional linear subspace $E\subset\R^3$, an affine function $\ell:E\to\R$, and six distinct points $y_1,\ldots,y_6\in\Sigma\cap E$ with $G_0(y_j)=\ell(y_j)$.
Since also $F_0(y_j)=c$, the map $(F_0,G_0)$ agrees at these points with the affine map $(c,\ell)$.
Hence the six points $(y_j,F_0(y_j),G_0(y_j))\in\R^3\times\R\times\R$ lie in the affine two-plane
\[
 \Pi=\{(y,c,\ell(y)):y\in E\}\subset\R^5.
\]

Finally, the affine map $\Psi:\R^5\to\R^{m+2}$ given by
\[
 \Psi(y,s,w)=\bigl(a+By,\ F(a)+\langle\nabla F(a),By\rangle+\sigma s,\ w\bigr)
\]
is injective, since $B$ is injective and $\sigma\neq0$.
As $\sigma^2=1$, it sends $(y,F_0(y),G_0(y))$ to $(a+By,F(a+By),G(a+By))$.
Therefore the six points $\Psi(y_j,c,\ell(y_j))$ are distinct points of the graph of $(F,G)$, and they lie in the affine two-plane $\Psi(\Pi)$.\qed

\begin{remark*}
The hypothesis $m\geq5$ is used only in Case~1, through the inequality $4<m$; Case~2 needs only $m\geq3$.
For $m=4$, the argument in Case~1 need not apply: the function $F(x)=x_1^2+x_2^2-x_3^2-x_4^2$ has a nonsingular Hessian with two positive and two negative eigenvalues at every point.
Thus the rank need not be less than $m$, while Case~2 is unavailable.
In that dimension, Theorem~\ref{thm:polygon} with $p=5$ gives five coplanar points instead.
\end{remark*}

{\raggedright

}
\thispagestyle{plain}

\begin{thebibliography}{12}

\bibitem{ChaperonMeyer}
M.~Chaperon and D.~Meyer, \emph{On a theorem of Ren\'e Thom in g\'eom\'etrie finie}, Enseign. Math. (2) \textbf{55} (2009), no.~3--4, 329--357.

\bibitem{Gadbled}
A.~Gadbled, \emph{Families of monotone symplectic manifolds constructed via symplectic cut and their Lagrangian submanifolds}, Algebr. Geom. Topol. \textbf{11} (2011), no.~4, 2319--2368.

\bibitem{Green}
B.~Green, \emph{100 open problems}, problem collection, \url{https://people.maths.ox.ac.uk/greenbj/papers/open-problems.pdf}, version consulted September 30, 2026.

\bibitem{Hatcher}
A.~Hatcher, \emph{Algebraic topology}, Cambridge University Press, Cambridge, 2002.

\bibitem{HatcherVB}
A.~Hatcher, \emph{Vector bundles and K-theory}, version 2.2, 2017, \url{https://pi.math.cornell.edu/~hatcher/VBKT/VB.pdf}.

\bibitem{KarasevVolovikov}
R.~N. Karasev and A.~Yu. Volovikov, \emph{Knaster's problem for almost $(\mathbb Z_p)^k$-orbits}, Topology Appl. \textbf{157} (2010), no.~5, 941--945.

\bibitem{Makeev1995}
V.~V. Makeev, \emph{Affine-inscribed and affine-circumscribed polygons and polytopes}, J. Math. Sci. (New York) \textbf{91} (1998), no.~6, 3518--3525.

\bibitem{Makeev2004}
V.~V. Makeev, \emph{Approximation of two-dimensional cross-sections of convex bodies by disks and ellipses}, St. Petersburg Math. J. \textbf{16} (2005), no.~6, 1043--1049.

\bibitem{Makeev2007}
V.~V. Makeev, \emph{Some extremal problems for vector bundles}, St. Petersburg Math. J. \textbf{19} (2008), no.~2, 261--277.

\bibitem{Makeev}
V.~V. Makeev, \emph{The Knaster problem and almost spherical sections}, Math. USSR-Sb. \textbf{66} (1990), no.~2, 431--438.

\bibitem{Tabachnikov}
S.~Tabachnikov, \emph{Converse Sturm--Hurwitz--Kellogg theorem and related results}, J. Fixed Point Theory Appl. \textbf{3} (2008), no.~1, 121--130.

\bibitem{Thom}
R.~Thom, \emph{Sur les vari\'et\'es d'ordre fini}, in: Global Analysis (Papers in Honor of K.~Kodaira), Univ. Tokyo Press, Tokyo, 1969, pp.~397--401.

\end{thebibliography}
\end{document}